\documentclass[12pt,leqno,fleqn]{amsart}  
\usepackage{amsmath,amstext,amsthm,amssymb,amsxtra}

\usepackage{txfonts} % pxfonts txfonts 
\usepackage[T1]{fontenc}
\usepackage{lmodern}

 \usepackage{euler}   % better than the option below

\usepackage{mathtools}
\mathtoolsset{showonlyrefs,showmanualtags} 

\usepackage[pagebackref,hypertexnames=false]{hyperref} %,hypertexnames=false,colorlinks,[pagebackref]

\usepackage[backrefs]{amsrefs}

\allowdisplaybreaks
\swapnumbers

\allowdisplaybreaks
\swapnumbers

\theoremstyle{plain} % definition 
\newtheorem{lemma}[equation]{Lemma} 
\newtheorem{proposition}[equation]{Proposition} 
\newtheorem{theorem}[equation]{Theorem} 
\newtheorem{corollary}[equation]{Corollary} 
\newtheorem{conjecture}[equation]{Conjecture}

\theoremstyle{definition}
\newtheorem{definition}[equation]{Definition} 

\theoremstyle{remark}

\newtheorem*{ack}{Acknowledgment}

\numberwithin{equation}{section}

\def\norm#1.#2.{\lVert#1\rVert_{#2}}
\def\Norm#1.#2.{\bigl\lVert#1\bigr\rVert_{#2}}
\def\NOrm#1.#2.{\Bigl\lVert#1\Bigr\rVert_{#2}}
\def\NORm#1.#2.{\biggl\lVert#1\biggr\rVert_{#2}}
\def\NORM#1.#2.{\Biggl\lVert#1\Biggr\rVert_{#2}}

\def\ip#1,#2,{\langle #1,#2\rangle}
\def\Ip#1,#2,{\langle#1,#2\rangle}
\def\IP#1,#2,{\langle#1,#2\rangle}

\def\Abs#1{\bigl\lvert#1\bigr\rvert}
\def\ABs#1{\Bigl\lvert#1\Bigr\rvert}

\def\XXint#1#2#3{{\setbox0=\hbox{$#1{#2#3}{\int}$}
     \vcenter{\hbox{$#2#3$}}\kern-.5\wd0}}

\title {Sharp $ A_2$ Inequality for Haar Shift Operators}
\author[M.T. Lacey]{Michael T. Lacey}   %  can use \and  
\thanks{Research supported in part by NSF grant 0456611}
\address{ School of Mathematics, Georgia Institute of Technology, Atlanta GA 30332, USA}
\email {lacey@math.gatech.edu}

\author[S. Petermichl]{Stefanie Petermichl}
\address{Department of Mathematics, Universite de Bordeaux 1, France }
\email{stefanie@math.u-bordeaux1.fr}

\author[M.C. Reguera]{Maria Carmen Reguera}
\thanks{Research supported in part by NSF grant 0456611}
\address{ School of Mathematics, Georgia Institute of Technology, Atlanta GA 30332, USA}
\email {mreguera@math.gatech.edu}

\begin{document}

\begin{abstract}
As a Corollary to the main result of the paper we give a new proof of the inequality 
\begin{equation*}
\norm \operatorname T f. L ^{2} (w). \lesssim \norm w. A_2. \norm f. L ^{2} (w). \,, 
\end{equation*}
where $ \operatorname T$ is either the Hilbert transform \cite{MR2354322}, a Riesz transform \cite{MR2367098}, or the 
Beurling operator \cite{MR1894362}.  The weight $ w$ is non-negative, and the linear growth in the $ A_2$ characteristic 
on the right is sharp.   Prior proofs relied strongly on Haar shift operators \cite{MR1756958} and Bellman function 
techniques.  The new proof uses Haar shifts, and then uses an elegant `two weight $ \operatorname T1$ theorem' 
of Nazarov-Treil-Volberg \cite{MR2407233} to immediately identify relevant Carleson measure estimates, which are 
in turn verified using an appropriate corona decomposition of the weight $ w$.  
\end{abstract}

\maketitle

%%%%%%%%%%%%%%%%%%%%%%%%%%%%%% SECTION  SECTION SECTION
%%%%%%%%%%%%%%%%%%%%%%%%%%%%%% SECTION  SECTION SECTION 
\section{Introduction} %\label{s.}

We are interested in weighted estimates for singular integral operators, and cognate operators, 
with a focus on sharp estimates in terms of the $ A_p$ characteristic of the weight.  
In particular we give a new proof of the estimate of Petermichl \cite{MR2354322} 
\begin{equation}\label{e.Hilbert}
\norm \operatorname H f . L ^{2} (w). \lesssim \norm w. A_2. \norm f. L^2 (w).  \,, 
\end{equation}
where $ \operatorname H f (x) = \textup{p.v.} \int f (x-y)\; dy/y$ is the Hilbert transform.  
Petermichl's proof, as well as corresponding inequalities for the Beurling operator \cite{MR1894362} 
and the Riesz transforms \cite{MR2367098} have relied upon a Bellman function approach to the estimate 
for the corresponding Haar shift.  We also analyze the Haar shifts, but instead use a 
deep two-weight inequality of Nazarov-Treil-Volberg \cite{MR2407233}  as a way to quickly 
reduce the question to certain Carleson measure estimates.  The latter estimates are proved 
by using the usual Haar functions together with appropriate corona decomposition.  The linear growth 
in terms of the  $ A_2$ characteristic is neatly explained by this decomposition.  

Let us precede to the definitions. 

%%%%%%%%%%%%%%%%%%%%%%%%%%%%%%  DEFINITION DEFINITION DEFINITION
\begin{definition}\label{d.Ap}  For $ w$ a positive function (a weight) on $ \mathbb R ^{d}$ we define 
the $ A_p$ characteristic of $ w$ to be 
\begin{equation*}
\norm w. A_p. \coloneqq \sup _{Q}  \lvert  Q\rvert ^{-1}  \int _{Q} w \; dx  \cdot \Bigl[ 
\lvert  Q\rvert ^{-1} \int _{Q} w ^{-1/(p-1)} \; dx\Bigr] ^{p-1} \,, \qquad 1<p<\infty \,, 
\end{equation*}
where the supremum is over all cubes in $ \mathbb R ^{d}$. 
\end{definition}
%%%%%%%%%%%%%%%%%%%%%%%%%%%%%%  DEFINITION DEFINITION DEFINITION

The relevant conjecture concerning the behavior of singular integral operators on the spaces $ L ^{p} (w)$ is 

%%%%%%%%%%%%%%%%%%%%%%%%%%%%%% CONJECTURE CONJECTURE CONJECTURE
\begin{conjecture}\label{j.Ap}  For a smooth singular integral operator $ \operatorname T  $ which is bounded on $ L
^{2} (dx)$ we have the estimate  
\begin{equation} \label{e.TAp}
\norm \operatorname T f. L ^{p} (w) . \lesssim \norm w. A_p. ^{\alpha (p)} \norm f. L ^{p} (w).  \,, 
\qquad \alpha (p)=\max \{1,1/(p-1)\} \,. 
\end{equation}
\end{conjecture}
%%%%%%%%%%%%%%%%%%%%%%%%%%%%%% CONJECTURE CONJECTURE CONJECTURE

An extrapolation estimate \cites{MR1894362,MR2140200} shows that it suffices to prove this estimate 
for $ p=2$, which is the   case we consider in the remainder of this paper.  
Currently this estimate is known for the Hilbert transform, Riesz transforms and the Beurling operator, 
with the proof using in an essential way the so-called Haar shift operators.  
This proof will do so as well, but handle all Haar shifts at the same time.  

%%%%%%%%%%%%%%%%%%%%%%%%%%%%%%  DEFINITION DEFINITION DEFINITION
\begin{definition}\label{d.haar} By a \emph{Haar function} $ h_Q$ on a cube $ Q\subset \mathbb R ^{d}$, we mean 
any function which satisfies 
%%  ITEMIZE    
\begin{enumerate}
\item  $ h_Q$ is a function supported on $ Q$, and is constant on  
dyadic subcubes of $ Q$.
(That is, $ h_Q$ is in the linear span of the indicators of the  `children' of $ Q$.)
\item  $\norm h_Q. \infty . \le \lvert  Q\rvert ^{-1/2}$. (So $ \norm h_Q. 2. \le 1$.)
\item  $ \int _Q h_Q (x)\; dx =0$. 
\end{enumerate}
%% ENUMERATE

\end{definition}
%%%%%%%%%%%%%%%%%%%%%%%%%%%%%%  DEFINITION DEFINITION DEFINITION

%%%%%%%%%%%%%%%%%%%%%%%%%%%%%%  DEFINITION DEFINITION DEFINITION
\begin{definition}\label{d.well} We say that $ \operatorname T$ is a \emph{Haar shift operator 
of index $ \tau  $}  iff 
\begin{gather*}
\operatorname Tf = \sum _{Q\in \mathscr Q} 
\sum _{ \substack{Q',Q''\subset Q\\  2 ^{-\tau  d} \lvert  Q\rvert \le \lvert  Q'\rvert, \lvert  Q''\rvert   }} 
a _{Q',Q''} \ip f, h_{Q'}, h _{Q''} \,, 
\\
\lvert  a _{Q',Q''}\rvert \le \Biggl[ \frac { \lvert  Q'\rvert } {\lvert  Q\rvert } \cdot 
\frac { \lvert  Q''\rvert } {\lvert  Q\rvert }\Biggr] ^{1/2} \,. 
\end{gather*}
\end{definition}
%%%%%%%%%%%%%%%%%%%%%%%%%%%%%%  DEFINITION DEFINITION DEFINITION

The point of the conditions in the definition is that $ \operatorname T$ be not only an $ L ^{2}(dx)$ 
bounded operator, but that it also be a Calder\'on-Zygmund operator.  In particular, it should admit 
a weak-$ L ^{1} (dx)$ bound that depends only on the index $ \tau  $.  See Proposition~\ref{p.weakL1}.

%%%%%%%%%%%%%%%%%%%%%%%%%%%%%% THEOREM THEOREM THEOREM
\begin{theorem}\label{t.a2}  Let $ \operatorname T$ be a  Haar shift operator of index $ \tau $, and 
let $ w$ be an $ A_2$ weight.  We have the inequality 
\begin{equation}  \label{e.a2}
\norm \operatorname T . L ^{2} (w ) \mapsto L ^{2} (w).  \lesssim \norm w . A_2. 
\end{equation}
The implied constant depends only dimension $ d$ and the index $ \tau $ of the operator. 
\end{theorem}
%%%%%%%%%%%%%%%%%%%%%%%%%%%%%% THEOREM THEOREM THEOREM

We have  this Corollary:  

%%%%%%%%%%%%%%%%%%%%%%%%%%%%%% COROLLARY COROLLARY COROLLARY
\begin{corollary}\label{c.a2}  The inequalities \eqref{e.TAp} holds for the Hilbert transform, 
the Riesz transforms in any dimension $ d$, and the Beurling operator on the plane.  
\end{corollary}
%%%%%%%%%%%%%%%%%%%%%%%%%%%%%%  COROLLARY COROLLARY COROLLARY

As is well-known, these singular integral operators are obtained by appropriate averaging of the 
Haar shifts, an argument invented in \cite{MR1756958}, to address the Hilbert transform.  
For the Riesz transforms, see \cite{MR1964822}, and the Beurling transform, see \cite{MR2140200}. 
We also derive, as a corollary, the sharp 
$ A_2$ bound for Haar square functions.  We leave the details of this to the reader.

The starting point of our proof is a beautiful `two weight $T1 $ Theorem for Haar shifts' 
due to Nazarov-Treil-Volberg \cite{MR2407233}. 
We 
recall a version of this Theorem in \S~\ref{s.ntv}.  This Theorem supplies necessary and sufficient 
conditions for an individual Haar shift to satisfy a two-weight $ L^2$ inequality, with the conditions 
being expressed in the language of the $\operatorname T1 $ Theorem.  In particular, it neatly identifies 
three estimates that need to be proved, with two related to paraproduct estimates.  In fact, this step 
is well-known, and is taken up immediately in e.\thinspace g.\thinspace \cite{MR2354322}.  
We then check the paraproduct bounds for $ A_2$ weights in \S~\ref{s.initial} and \S~\ref{s.main}, which 
is the main new step in this paper.

The question of bounds for singular integral operators on $ L ^{p} (w)$ that are sharp with respect to the 
$ A_p$ characteristic was identified in an  influential 
paper of Buckley, \cite{MR1124164}.  It took many years to find the first proofs of such estimates.  
We refer  the reader to  \cite{MR2354322} for some of this history, and point to 
the central role of the work of Nazarov-Treil-Volberg \cite{MR1685781} in shaping much of the work cited here. 
The prior proofs of Corollary~\ref{c.a2} have all relied upon Bellman function techniques.  And indeed, this 
technique will supply a proof of the   results in this paper.  The Beurling operator is  the most easily 
available,  since this operator can be seen as the average of the simplest of 
Haar shifts, namely martingale transforms, see \cite{MR1992955}.  
The $ A_2$ bound was derived for Martingale  transforms by J.~Wittwer \cite{MR1748283}.
%(Alternate square functions were considered in \cite{MR1771755}.)
The paraproduct structure is much more central to the 
problem if one works with Haar shifts that pair a `parent' Haar with a `child' Haar. 
If one considers Square Functions, sharp results were obtained  in $ L ^2 $ by Wittwer \cite{MR1897458}, 
and Hukovic-Treil-Volberg \cite{MR1771755}. 
Recently, Beznosova \cite{MR2433959}, has proved the linear bound 
for discrete paraproduct operators, again using the Bellman function method. It would be of interest to prove 
her Theorem with techniques closer to those of this paper.

\begin{ack} The authors are participants in a research program at the Centre de Recerca Matem\'atica, 
at the Universitat Aut\`onoma Barcelona,  Spain.  We thank the Centre for their hospitality, and very supportive environment.  
Xavier Tolsa pointed out some relevant references to us. 
\end{ack}

%%%%%%%%%%%%%%%%%%%%%%%%%%%%%% SECTION  SECTION SECTION
%%%%%%%%%%%%%%%%%%%%%%%%%%%%%% SECTION  SECTION SECTION 
\section{The Characterization of Nazarov-Treil-Volberg} \label{s.ntv}

The success of this approach is based upon a beautiful characterization of two weight inequalities. 
Indeed, this characterization is true for \emph{individual} two-weight inequalities.  This Theorem can be 
thought of as a `Two Weight $ T1$ Theorem.'   We are stating only a sub-case of their Theorem, which does 
not assume that the operators satisfy an $ L ^{2} (dx)$ bound. 

%%%%%%%%%%%%%%%%%%%%%%%%%%%%%% THEOREM THEOREM THEOREM
\begin{theorem}\label{t.ntv} [Nazarov-Treil-Volberg \cite{MR2407233}] Let $ \operatorname T $ be 
a Haar shift operator of index $ \tau $, as in Definition~\ref{d.well},
and $ \sigma , \mu $ two 
positive measures.  The $ L ^2 $ inequality 
\begin{equation}\label{e.ntv1} 
\norm \operatorname T (\sigma f). L ^{2} (\mu ).  \lesssim \norm f. L ^{2} (\sigma ). 
\end{equation}
holds iff the following three conditions hold. 
For all cubes $Q, Q', Q''$ with $ Q', Q''\subset Q$ and $  2 ^{- (\tau -1) d} \lvert  Q\rvert \le  \lvert  Q'\rvert , 
\lvert  Q''\rvert $, 
\begin{align}\label{e.ntv2} 
\ABs{\int _{Q''}\operatorname T (\sigma \mathbf 1_{Q'}) \; \mu (dx) }
&\le C _{\textup{WB}} \sqrt {\sigma (Q') \mu  (Q'')}  & \textup{(Weak Bnded)}
\\ \label{e.ntv3}
\norm \operatorname T(\sigma \mathbf 1_{Q}). L ^{2} (Q, \mu ). & \le C _{T1} \sqrt {\sigma (Q)} 
& (T1\in BMO) 
\\ \label{e.ntv4}
\norm \operatorname T ^{\ast }(\mu  \mathbf 1_{Q}). L ^{2} (Q, \sigma  ). & \le C _{T ^{\ast }1} \sqrt {\mu (Q)} 
& (T ^{\ast }1\in BMO)  
\end{align}
Moreover, we have the inequality 
\begin{equation}\label{e.ntv5}
\norm \operatorname T (\sigma \cdot ). L ^{2} (\sigma )\to L ^{2} (\mu ). 
\lesssim C _{\textup{WB}}+ C _{T 1} + C _{T ^{\ast }1}  \,. 
\end{equation}
\end{theorem}
%%%%%%%%%%%%%%%%%%%%%%%%%%%%%% THEOREM THEOREM THEOREM

This Theorem is contained in \cite{MR2407233}*{Theorem 1.4}, aside from the claim \eqref{e.ntv5}.  
But this inequality can be seen from the proof in their paper.  Indeed, their proof is in close analogy to the 
$ T1$ Theorem.  Briefly, the proof is as follows.  The operator $ \operatorname T (\sigma \cdot )$ 
is expanded in `Haar basis', but  
the Haar bases are 
 adapted to the two measures $ \sigma $ and $ \mu $. 
This technique appeared in  \cite{MR1685781}, and has been used  subsequently in 
\cites{MR1992955,MR1748283,MR2354322}. 
  Expressing the bilinear form $ \int \operatorname T (\sigma f) \cdot g 
\; \mu  $ as a matrix in these two bases, the 
matrix is split into three parts.  Those terms `close to the diagonal' are controlled by the 
`weak boundedness' condition \eqref{e.ntv2}.  Those terms below and above the diagonal are 
recognized as paraproducts.  One of these is of the form 
\begin{equation} \label{e.P}
\operatorname P (f) \coloneqq 
\sum _{Q}  \sigma (Q) ^{-1} \int _{Q} f \sigma \; dy \cdot  \Delta _{Q} ^{w} ( \operatorname T (\sigma 1))
\end{equation}
Here the first term is an average of $ f$ with respect to the measure $ \sigma $, and the second is 
a martingale difference of $ \operatorname T (\sigma 1)$ with respect to the measure $ w$.  
In particular, $ \Delta _{Q} ^{w} ( \operatorname T (\sigma 1))$ are $ w$-orthogonal functions in $ Q$. 
Thus, one has the equality 
\begin{equation*}
\norm \operatorname P (f). L ^{2 } (w) . ^2 
= \sum _{Q} \Abs{\sigma (Q) ^{-1} \int _{Q} f \sigma \; dy}^2  \cdot 
\norm \Delta _{Q} ^{w} ( \operatorname T (\sigma 1)). L ^{2} (w). ^2 \,. 
\end{equation*}
The inequality $ \norm \operatorname P (f). L ^{2 } (w) . \lesssim \norm f. L ^2 (\sigma ). $ 
is a weighted Carleson embedding inequality that is implied by  the `$ \operatorname T1 \in BMO$' condition 
\eqref{e.ntv3}.  The other paraproduct is dual to the one in \eqref{e.P}.

%%%%%%%%%%%%%%%%%%%%%%%%%%%%%% SECTION  SECTION SECTION
%%%%%%%%%%%%%%%%%%%%%%%%%%%%%% SECTION  SECTION SECTION
\section{Initial Considerations} \label{s.initial}

We collect together a potpourri of facts that will be useful to us, and are of somewhat general nature. 
We begin with a  somewhat complicated definition that we will use in order to organize the proof of our main 
estimate.

%%%%%%%%%%%%%%%%%%%%%%%%%%%%%%  DEFINITION DEFINITION DEFINITION
\begin{definition}\label{d.linearizing}Let $ \mathscr Q'\subset  \mathscr Q$ be any 
 collection of dyadic cubes, and $ \mu $ a positive measure. 
 Call $ (\mathscr L \;:\; \mathscr Q' (L) ) $ 
a \emph{	$ \mu $-corona decomposition of  $ \mathscr Q'$ }  
 if these conditions hold. 
%%  ENUMERATE
\begin{enumerate}
\item For each $ Q\in \mathscr Q'$ there is a member of $ \mathscr L$ that contains 
 $ Q$, and letting $ \lambda (Q)\in \mathscr L$ denote the minimal cube which contains $ Q$ we have 
 \begin{equation}
\label{e.lin2}
4\frac {\mu(\lambda (Q))} {\lvert  \lambda (Q)\rvert } \ge   \frac {\mu(Q)} {\lvert  Q\rvert }  \,. 
\end{equation}
\item For all $ L \subsetneq L' \in \mathscr L$
\begin{equation}
\label{e.lin1}
\frac {\mu(L)} {\lvert  L\rvert } > 4 \frac {\mu(L')} {\lvert  L'\rvert } \,. 
\end{equation}
\end{enumerate}
%% ENUMERATE
We set  $ \mathscr Q' (L) \coloneqq \{ Q\in \mathscr Q' \;:\; \lambda (Q)= L\}$. 
The collections $ \mathscr Q' (L)$ partition $ \mathscr Q'$. 
\end{definition}
%%%%%%%%%%%%%%%%%%%%%%%%%%%%%%  DEFINITION DEFINITION DEFINITION

Decompositions of this type appear in a variety of questions.  We are using terminology which goes back 
to (at least) David and Semmes \cites{MR1251061,MR1113517}, though the same  type of construction 
appears as early as 1977 in \cite{MR0447956}, where it is called the `principle cube' construction. 
A subtle corona decomposition is central to \cite{MR2179730}, and the paper \cite{MR1934198} includes several 
examples in the context of dyadic analysis.

A basic fact  is this.
\begin{equation} \label{e.11/4}
\ABs{\bigcup _{\substack{L'\in \mathscr L\\ L'\subsetneq L }}  L'} \le \tfrac 14 \lvert  L\rvert\,, 
\qquad L\in \mathscr L 
\,.   
\end{equation}
This follows from \eqref{e.lin1},  which says that the intervals $ L' \subset L$ have much more than their fair 
share of the mass of $ \mu $, hence the $ L'$ have to be smaller intervals.  And this easily implies 
\begin{equation}\label{e.1/4}
\NOrm \sum _{\substack{L'\in \mathscr L\\ L'\subset L }} \mathbf 1_{L'} . 2. \lesssim \lvert  L\rvert ^{1/2} \,.  
\end{equation}

We have the following (known) Lemma,  but we detail it as it is one way that the $ A_2$ condition 
enters in the proof. 

%%%%%%%%%%%%%%%%%%%%%%%%%%%%%% LEMMA LEMMA LEMMA
\begin{lemma}\label{l.carleson} 
Let $ \mathscr L$ be associated with corona decomposition  for am $ A_2$ weight $ w$.  For any cube $ Q$ we have 
\begin{equation}\label{e.carleson}
\sum _{\substack{ L\in \mathscr L\\ L\subset Q}} w (L) \le \tfrac {16}9 \norm w. A_2.  w (Q) \,. 
\end{equation}
\end{lemma}
%%%%%%%%%%%%%%%%%%%%%%%%%%%%%% LEMMA LEMMA LEMMA

%%%%%%%%%%%%%%%%%%%%%%%%%%%%%% PROOF PROOF PROOF
\begin{proof}  It suffices to show this: 
For  $ L\in \mathscr L $ 
\begin{equation} \label{e.k}
w \bigl( \textstyle\bigcup\{ L'\in \mathscr L \;:\; L'\subsetneq   L\}\bigr) \le 
(1-c \norm w. A_2. ^{-1} )  w (L)\,, \qquad c= \frac 9 {16} \,. 
\end{equation}

We begin with a calculation related to $ A _{\infty }$. 
Let $ E$ be a measurable subset of $ L$.  Then, 
\begin{align}
\frac {\lvert  E\rvert } {\lvert  L\rvert } & = \lvert  L\rvert ^{-1} \int _{E} w ^{1/2} \cdot w ^{-1/2} \; dx
\\
& \le \Biggl[ \frac {w (E)} {\lvert  L\rvert } \cdot \frac {w ^{-1} (L)} {\lvert  L\rvert } \Biggr] ^{1/2} 
\\  \label{e.AzI}
& \le
\Biggl[   \norm w. A_2. \frac {w (E)} {w (L)} \Biggr] ^{1/2} \,. 
\end{align}

Apply this with 
$ L-E= \bigcup \{ L'\in \mathscr L \;:\; L'\subsetneq   L\}$.  Then, by \eqref{e.11/4}, $ \lvert  L-E\rvert< \frac 14 \lvert
L\rvert  $, so that $ \lvert  E\rvert\ge \tfrac 34 \lvert  L\rvert $.
It follows that we then have 
\begin{equation*}
\frac 9 {16 \norm w. A_2.} \cdot w (L) \le w (E)\,.  
\end{equation*}
Whence, we see that \eqref{e.k} holds. 
Our proof is complete. 
\end{proof}
%%%%%%%%%%%%%%%%%%%%%%%%%%%%%% PROOF PROOF PROOF

Concerning the Haar shift operators $ \operatorname T$, we make the following definition. 

%%%%%%%%%%%%%%%%%%%%%%%%%%%%%%  DEFINITION DEFINITION DEFINITION
\begin{definition}\label{d.simple} 
 We say that $ \operatorname T$ is a \emph{simple Haar shift operator 
of index $ \tau  $}  iff  
\begin{gather}
\operatorname Tf = \sum _{Q\in \mathscr Q} 
 \ip f, g_{Q}, \gamma  _{Q} \,,  
\\  \label{e.cancellation}
g_Q, \gamma _Q \in \operatorname {span} ( h_{Q'} \;:\; Q'\subset Q\,, 2 ^{-\tau d} \lvert  Q\rvert \le \lvert  Q'\rvert
) \,, 
\\ \label{e.size}
\norm g_Q.\infty . \,,\, \norm \gamma _Q . \infty . \le \lvert  Q\rvert ^{-1/2} \,.   
\end{gather}
\end{definition}
%%%%%%%%%%%%%%%%%%%%%%%%%%%%%%  DEFINITION DEFINITION DEFINITION
Below, we will only consider simple Haar shift operators. 
The  important property they satisfy is

%%%%%%%%%%%%%%%%%%%%%%%%%%%%%% PROPOSITION PROPOSITION PROPOSITION
\begin{proposition}\label{p.weakL1}  A simple Haar shift operator $ \operatorname T$ with index $ \tau $ maps 
$ L ^{2}(dx)$ into itself with norm at most $ \lesssim \tau $. It maps 
$ L ^{1}(dx)$ 
into $ L ^{1,\infty } (dx)$ with norm $ \lesssim 2 ^{\tau d} $. 
\end{proposition}
%%%%%%%%%%%%%%%%%%%%%%%%%%%%%% PROPOSITION PROPOSITION PROPOSITION

The point is that these bounds only depend upon the index $ \tau $.

%%%%%%%%%%%%%%%%%%%%%%%%%%%%%% PROOF PROOF PROOF
\begin{proof}
The proof is well-known, but we present it as some similar difficulties appear later in the proof; see the 
discussion following \eqref{e.qx}. 
Set 
\begin{equation*}
\operatorname T _ s  f 
\coloneqq  \sum _{\substack{Q\in \mathscr Q\\  \lvert  Q\rvert= 2 ^{s d}  }}  \ip f, g_{Q}, \gamma  _{Q} \,,  
\end{equation*}
which is the operator at scale $ 2 ^{s}$.  The `size condition' \eqref{e.size} implies that 
$ \norm \operatorname T_s. L ^{2} (dx).  \le 1$.  The `cancellation condition' \eqref{e.cancellation} 
then implies that 
\begin{equation*}
\operatorname T _{s}\operatorname T ^{\ast } _{s'}
=\operatorname T ^{\ast } _{s}\operatorname T _{s'}=0  \,, \qquad \lvert  s-s'\rvert> \tau \,.   
\end{equation*}
So we see that $ \norm \operatorname T. L ^{2} (dx). \le \tau +1$. 

Concerning the weak $ L ^{1} (dx)$ inequality, we use the usual proof.
Fix $ f\in L ^{1} (dx)$. Apply the dyadic Calder\'on-Zygmund Decomposition to $ f$ 
at height $ \lambda $.  Thus, $ f= g+b$ where $ \norm g. 2. \lesssim \sqrt \lambda  \norm f.L^1 (dx). ^{1/2} $, 
and $ b$ is supported on a union of  disjoint dyadic cubes $ Q \in \mathscr B$ with 
\begin{gather} \label{e.bad0}
\int _Q b \; dx =0 \,, \qquad Q\in \mathscr B\,, 
\\ \label{e.badb}
\sum _{Q\in \mathscr B} \lvert  Q\rvert \lesssim \lambda ^{-1} \norm f.1. \,.   
\end{gather}

For the `good' function $ g$, using the $ L ^{2} (dx)$ estimate we have 
\begin{align}
\lvert  \{ \operatorname T g>  \tau \lambda \}\rvert & \le (\tau \lambda) ^{-2} \norm \operatorname T g. L ^{2} (dx). ^2 
\\ \label{e.Good}
& \lesssim   \lambda ^{-2} \norm g. 2. ^2 \lesssim   \lambda ^{-1} \norm f.L^1 (dx). \,.
\end{align}

For the `bad' function, we modify the usual argument.  For a dyadic cube $ Q$,  and integer $ t$, let $ Q ^{(t)}$
denote it's $ t$-fold parent.  Thus, $ Q ^{(1)}$ is the minimal dyadic cube that strictly contains $ Q$, and 
inductively, $ Q ^{(t+1)}= (Q ^{(t)}) ^{(1)}$.  Observe that  \eqref{e.badb} implies  
\begin{equation} \label{e.parents}
\Abs{ \bigcup \{ Q ^{(\tau )} \;:\; Q \in \mathscr B\}} \lesssim 2 ^{\tau  d} \lambda ^{-1} \norm f.1. \,. 
\end{equation}
And, the 'cancellation condition' \eqref{e.cancellation}, with \eqref{e.bad0}, imply that 
for $ Q\in \mathscr B$, and $ x\not\in Q ^{(\tau )}$, we have 
\begin{equation*}
\operatorname T (\mathbf 1_{Q} b) (x) 
= \sum _{Q' \;:\; Q ^{(\tau )} \subsetneq Q'} 
\ip \mathbf 1_{Q} b, g_{Q'}, \gamma _{Q'} (x) = 0 
\end{equation*}
since $ g_{Q'}$ will be constant on the cube $ Q$.

Hence, we have 
\begin{align*}
\lvert  \{\operatorname T (b) > \lambda \}\rvert 
\le 
\Abs{ \bigcup \{ Q ^{(\tau )} \;:\; Q \in \mathscr B\}} \lesssim 2 ^{\tau  d} \lambda ^{-1} \norm f.1. \,. 
\end{align*}
This completes the proof. 

\end{proof}
%%%%%%%%%%%%%%%%%%%%%%%%%%%%%% PROOF PROOF PROOF

We need  a version of the John-Nirenberg inequality, which
says that a `uniform $ L ^{0}$ condition implies exponential integrability.'  

%%%%%%%%%%%%%%%%%%%%%%%%%%%%%% LEMMA LEMMA LEMMA
\begin{lemma}\label{l.JN}  This holds for all integers $ \tau $. 
Let $ \{\phi _Q \;:\; Q\in \mathscr Q\} $ be functions so that for all dyadic cubes $ Q$ we have 
%%  ENUMERATE
\begin{enumerate}
\item $ \phi _Q$ is supported on $ Q$ and is constant on each sub-cube $ Q'\subset Q$ with $ \lvert  Q'\rvert= 2 ^{-\tau d} 
\lvert  Q\rvert $; 
\item $\norm \phi_Q. \infty . \le 1 $; 
\item  
for all dyadic cubes $ Q$, we have 
\begin{equation}\label{e.JN1}
\ABs{ \Biggl\{  \ABs{\sum _{Q' \;:\; Q'\subset Q} \phi _{Q'} } > 1  \Biggr\}} \le 2 ^{-\tau d-1} \lvert  Q\rvert \,.  
\end{equation}
\end{enumerate}
%% ENUMERATE
It then follows that we have the estimate uniform in $ Q$ and $ t>1$.  
\begin{equation}\label{e.JN2}
\ABs{ \Biggl\{  \ABs{\sum _{Q' \;:\; Q'\subset Q} \phi _{Q'} } > 2 \tau  t  \Biggr\}} \le  \tau  2  ^{-t+1} \lvert  Q\rvert  \,, 
\qquad t >1 \,. 
\end{equation}
\end{lemma}
\section{The Main Argument} \label{s.main}
We begin the main line of argument to prove \eqref{e.a2}. We no longer try to keep track of the dependence on 
$ \tau $ in our estimates. (It is, in any case, exponential in $ \tau $.)  Accordingly, we assume that 
we work with a subset $ \mathscr Q _{\tau } $ of dyadic cubes with `scales separated by $ \tau $.'  That is, 
we assume that for $ Q'\subsetneq Q $ and $ Q',Q\in \mathscr Q$ we have $ \lvert  Q'\rvert\le 2 ^{-d\tau} \lvert
Q\rvert  $, where $ d$ is dimension.

It is well-known that \eqref{e.a2} is equivalent 
to showing that 
\begin{equation*}
\norm \operatorname T (fw). L ^{2} (w ^{-1} ). \lesssim \norm w. A_2. \norm f. L ^{2} (w). \,. 
\end{equation*}
Here we are using the dual-measure formulation, so that the measure $ w$ appears on both sides of the inequality, 
as in Theorem~\ref{t.ntv}.

By Theorem~\ref{t.ntv}, and the symmetry of the $ A_2$ condition, 
it is sufficient to check that   the two inequalities below hold for all simple 
Haar shift operators $ \operatorname T$ of index $ \tau $: 
\begin{align}\label{e.I2}
\lvert  \ip \operatorname T ( w\mathbf 1_{Q}), w ^{-1} \mathbf 1_{R},  \rvert 
&\lesssim  \norm w . A_2. \sqrt {w (Q) w ^{-1} (R)}\,, 
\\ \label{e.I3}
\int _{Q} \lvert  \operatorname T (w \mathbf 1_{Q})\rvert ^2 \; w ^{-1} dx 
&\lesssim \norm w. A_2. ^2 w (Q) \,. 
\end{align}
These should hold for all dyadic cubes $ Q$, and 
in \eqref{e.I2}, we have $  2 ^{-(\tau+1) d}\lvert  Q\rvert \le \lvert  R\rvert\le 2 ^{(\tau+1) d} \lvert  Q\rvert  $.

In the present circumstance, the `weak boundedness' inequality \eqref{e.I2} can be derived from the `$ T1$' inequality 
\eqref{e.I3}.  We can assume that $ \lvert  Q\rvert\le \lvert  R\rvert  $ by passing to the dual operator 
and replacing $ w$ by $ w ^{-1} $. If $ \lvert  Q\rvert=\lvert  R\rvert  $, the inner product is zero unless 
$ Q=R$. But then we just appeal to \eqref{e.I3}. 
\begin{align*}
\lvert  \ip \operatorname T ( w\mathbf 1_{Q}), w ^{-1} \mathbf 1_{Q},  \rvert
& \le \sqrt {w ^{-1} (Q)} \cdot \norm \mathbf 1_{Q} 
\operatorname T ( w\mathbf 1_{Q}) . L ^{2} (w ^{-1} ). 
\\
& \lesssim \norm w.A_2. \sqrt {w (Q) \cdot w ^{-1} (Q)} \,. 
\end{align*}

If $ \lvert  Q\rvert< \lvert  R\rvert  $, let assume that $ Q\subset R$, and write 
\begin{align*}
\lvert  \ip \operatorname T ( w\mathbf 1_{Q}), w ^{-1} \mathbf 1_{R},  \rvert 
& \le 
\lvert  \ip \operatorname T ( w\mathbf 1_{Q}), w ^{-1} \mathbf 1_{Q},  \rvert 
+
\lvert  \ip \operatorname T ( w\mathbf 1_{Q}), w ^{-1} \mathbf 1_{R-Q},  \rvert \,.  
\end{align*}
The first term on the right is handled just as in the previous case.  In the second case, 
we use the fact that $ 2 ^{-\tau d}\lvert  R\rvert\le \lvert  Q\rvert< \lvert  R\rvert   $, 
so that there is a difference in scales between the two cubes of only at most $ \tau $ scales. 
That, with the size conditions on $ \operatorname T$ lead to 
\begin{align}
\lvert  \ip \operatorname T ( w\mathbf 1_{Q}), w ^{-1} \mathbf 1_{R-Q},  \rvert 
& \lesssim 
\frac {w (Q) w ^{-1} (R)} {\lvert  R\rvert } 
\\ \label{e.largescales}
& \lesssim \norm w. A_2. \sqrt {w (Q) \cdot w ^{-1} (R)} \,. 
\end{align}
The last inequality follows since 
\begin{align*}
\sqrt {\frac {w (Q) w ^{-1} (R)} {\lvert  R\rvert ^2 } } 
& 
\le 
\sqrt {\frac {w (R) w ^{-1} (R)} {\lvert  R\rvert ^2 } } 
 \le \sqrt {\norm w. A_2.} \le \norm w.A_2. \,.  
\end{align*}
Indeed, we always have $ 1\le \norm w.A_2.$.  The case of $ Q\cap R= \emptyset $ is handled in a similar 
fashion.  

\bigskip 
To verify \eqref{e.I3}, we first treat the  `large scales.'
\begin{align*}
\NOrm  \mathbf 1_{Q_0 } \sum _{Q \;:\; Q\supsetneq Q_0} \ip w \mathbf 1_{Q_0}, g_{Q}, \gamma _Q . L ^{2} (w ^{-1} ).
& 
\lesssim \frac {w (Q_0) w ^{-1} (Q_0) ^{1/2}  } {\lvert  Q_0\rvert }
\\
& \lesssim \sqrt {w (Q_0)} \cdot \norm w.A_2. 
\end{align*}

Therefore, it suffices to prove 
\begin{equation}\label{e.I33}
\NOrm \sum _{Q \;:\; Q\subset Q_0} \ip w, g_{Q}, \gamma _Q . L ^{2} (w ^{-1} ). \lesssim \norm w. A_2. \sqrt {w (Q_0)} \,. 
\end{equation}

Let us define for dyadic cubes $ Q_0$ and collections of dyadic cubes $ \mathscr Q'$,
\begin{align}\label{e.HQ}
H (Q_0,\mathscr Q')&  \coloneqq \sum _{\substack{Q\subset Q_0\\ Q\in \mathscr Q' }}
\ip w, g_{Q}, \gamma _Q \,, 
\\ \label{e.HHQ}
\mathbf H (\mathscr Q') & \coloneqq \sup _{Q_0} \frac {\norm H (Q_0, \mathscr Q') . L ^{2} (w ^{-1} ).} {\sqrt {w (Q_0)}} \,. 
\end{align}
It is a useful remark that in estimating $ \mathbf H (\mathscr Q')$ we 
can restrict the supremum  to cubes $ Q_0\in \mathscr Q'$. 
Of course, we are seeking to prove $ \mathbf H (\mathscr Q) \lesssim \norm w. A_2. $.

\medskip

The first important definition here is 
\begin{equation} \label{e.QnDef}
\mathscr Q_n \coloneqq \Biggl\{Q\in \mathscr Q  \;:\; 
2 ^{n-1}< \frac {w (Q)} {\lvert  Q\rvert } \cdot \frac {w ^{-1} (Q)} {\lvert  Q\rvert } \le 2 ^{n}
\Biggr\}\,. 
\end{equation}
We show that 
\begin{equation}\label{e.Qn}
\mathbf H (\mathscr Q_n) \lesssim 2 ^{n/2} \norm w. A_2. ^{1/2}  \,. 
\end{equation}
Since $ 2^n\le \norm w. A_2.$, this estimate is summable in $ n $ to prove \eqref{e.I33}.

\medskip

Now fix a $ Q_0\in \mathscr Q_n$ for which we are to test the supremum in \eqref{e.HHQ}.  
Let $ \mathscr P_n= \{Q\in \mathscr Q_n \;:\; Q\subset Q_0\}$.  Let $ (\mathscr L_n \;:\; \mathscr P_n (L))$ be a 
corona decomposition of $ \mathscr P_n$ relative to measure $ w $. 
 (The reader is advised to recall the  Definition~\ref{d.linearizing}.)

\medskip 
The essence of the matter is contained in the following Lemma. 
%%%%%%%%%%%%%%%%%%%%%%%%%%%%%% LEMMA LEMMA LEMMA
\begin{lemma}\label{l.essence}  We have these  distributional estimates, uniform over $ L\in \mathscr L_n$: 
\begin{align}\label{e.ess1} 
\Abs{
\bigl\{ 
x \in L \;:\; \lvert  H (L, \mathscr P_n (L)) (x) \rvert > 
K t \tfrac {w (L)} {\lvert  L\rvert }  
\bigr\}} &\lesssim \operatorname e ^{-t} \lvert  L\rvert\,,  
\\ \label{e.ess2}
w ^{-1} \bigl( 
\bigl\{ 
x \in L \;:\; \lvert  H (L, \mathscr P_n (L)) (x) \rvert > 
K t \tfrac {w (L)} {\lvert  L\rvert }  
\bigr\}\bigr) &\lesssim \operatorname e ^{-t}  w ^{-1} (L)\,.   
\end{align}
\end{lemma}
%%%%%%%%%%%%%%%%%%%%%%%%%%%%%% LEMMA LEMMA LEMMA

Let us complete the proof of our Theorem based upon this Lemma.  
Set $  H _{n} (L) \coloneqq \lvert H(L, \mathscr P_n (L))\rvert$, and estimate 
\begin{align}
\norm H (Q_0, \mathscr Q_n) . L ^{2} (w ^{-1} ). ^2 
& \le 
\NOrm \sum _{L\in \mathscr L_n }   H_n(L) . L ^{2} (w ^{-1} ). ^2 
\\  \label{e.A+B}
& = A+2B= A+ 2\sum _{L\in \mathscr L_n } B (L) \,, 
\\
A& \coloneqq  \sum _{L\in \mathscr L_n } \norm  H_n(L).  L ^{2} (w ^{-1} ). ^2  
\\ \label{e.B(L)}
B (L) & \coloneqq 
\sum _{ \substack{L'\in \mathscr L_n\\ L'\subsetneq L }} \int 
   H_n(L) \cdot H_n(L') \; w ^{-1} \,. 
\end{align}
Note that these estimates show that all cancellation necessary for the truth of theorem is already captured in the corona decomposition.

The estimate of $ A$ is straight forward.  By \eqref{e.ess2},  we see that the 
$ A_2$ estimate reveals itself.  
\begin{align*}
\norm H_n(L). L ^{2} (w ^{-1} ). ^2 & \lesssim    \Bigl[ \frac {w (L)} {  \lvert  L \rvert  } \Bigr] ^2 w ^{-1} (L) 
\\    
& \lesssim w (L) \frac {w (L)} {  \lvert L\rvert } \cdot \frac {w ^{-1}  (L)} {  \lvert L\rvert }
\\
& \lesssim  2 ^{n} w (L) \,. 
\end{align*}
Therefore, by \eqref{e.carleson} 
\begin{align}\label{e.A}
A \lesssim  2 ^{n}\sum _{L\in \mathscr L_n } w (L) \lesssim  2 ^{n} \norm w. A_2. w (Q_0) \,. 
\end{align}

\medskip 

In the expression \eqref{e.B(L)}, the integral is not as complicated as it  immediately appears. 
We have assumed that 
'scales are separated by $ \tau $' at the beginning of this section,  so that as $ L'$ is strictly contained in $ L$, 
we have  for any $ Q\in \mathscr P_n(L)$, that $ (L') ^{(\tau )} $ is either contained in $ Q$ or disjoint from it. 
It follows that $ H_n(L) $ takes a single value on all of $ L'$, 
which we denote by $  H_n(L;L')$.  This observation 
simplifies our task of estimating the integral.  

For $ L' \subsetneq L $ we use \eqref{e.ess2} and \eqref{e.QnDef} to see that 
\begin{align}
\int 
   H_n(L) \cdot H_n(L') \; w ^{-1}
&\lesssim 
   H_n(L;L')  \frac {w (L')} {  \lvert L'\rvert } \cdot w ^{-1} (L') 
\\&  \label{e.gt}
\lesssim 2 ^{n}   H_n(L;L')  \cdot   \lvert L'\rvert \,.  
\end{align}
Note that the $ A_2$ characteristic has entered in.   And the presence of $   \lvert L'\rvert $ indicates 
that there is an integral against Lebesgue measure here.

Employ this observation with Cauchy-Schwartz, \emph{both} distributional estimates \eqref{e.ess1} 
and \eqref{e.ess2}  as well as  \eqref{e.1/4}  to estimate 
\begin{align}\label{e.zj}
B (L) & \coloneqq 
\sum _{\substack{ L' \in \mathscr L_n  \\ L'\subsetneq L }} 
\int 
   H_n(L) \cdot H_n(L') \; w ^{-1}
\\
& \lesssim 
2 ^{n}
   H_n(L;L')   
\sum _{\substack{ L' \in \mathscr L_n  \\L'\subsetneq  L }} 
  \lvert  L'\rvert    & (\textup{by \eqref{e.gt}})
\\
& = 2 ^{n} 
\int    H_n(L;L')  \cdot \sum _{\substack{ L' \in \mathscr L_n  \\ L'\subset L }}  \mathbf 1_{L'} \; dx 
& (\textup{by defn.})
\\
& \le 2 ^{n} \norm  H_n(L). L ^{2} (dx). \NOrm \sum _{\substack{ L' \in \mathscr L_n  \\ L'\subset L }}
\mathbf 1_{L'} 
. L ^{2} (dx).   & (\textup{Cauchy-Schwartz})
\\
& \lesssim 2 ^{n}w (L) \,. & (\textup{by \eqref{e.ess1} and \eqref{e.1/4}}) 
\end{align}
Therefore, by \eqref{e.carleson} again, 
\begin{align*}
B  & \lesssim 2 ^{n} \sum _{L\in \mathscr L_n} w (L) \lesssim 2 ^{n} \norm w. A_2.  w (Q_0) \,. 
\end{align*}
Combining this estimate with \eqref{e.A+B} and \eqref{e.A} completes the proof of \eqref{e.Qn}, and so our 
Theorem, assuming Lemma~\ref{l.essence}. 

%%%%%%%%%%%%%%%%%%%%%%%%%%%%%% SECTION  SECTION SECTION
%%%%%%%%%%%%%%%%%%%%%%%%%%%%%% SECTION  SECTION SECTION 
\section{The essence of the matter.} %\label{s.}

We prove Lemma~\ref{l.essence}. 
In this situation, both a cube $ Q_0$ and cube $ L\in \mathscr L_n$ are given.  
It is an important point that all the relevant 
cubes that we sum over are in the collection $ \mathscr Q_n$, as defined in \eqref{e.QnDef}.

One more class of dyadic cubes are needed.  
For integers $ \alpha \ge0$ define $ \mathscr P_{n,\alpha}(L) $ to be those $ Q\in \mathscr P_n (L)$  
such that 
\begin{equation}\label{e.al}
 2 ^{-\alpha +1} \frac {w (L)} { \lvert  L\rvert } \le 
 \frac {w (Q)} {\lvert  Q\rvert } < 2 ^{-\alpha +2 } \frac {w (L)} { \lvert  L\rvert } \,. 
\end{equation}

The essential observation is this: By Proposition~\ref{p.weakL1},
$ \operatorname T $ maps $ L ^{1} (dx)$ into weak-$ L ^{1} (dx)$, 
with norm depending only on the index $ \tau $ of the operator.  Hence, 
\begin{equation}
\NOrm \sum _{\substack{Q \subset Q_1\\ Q\in \mathscr P_{n,\alpha}(L)} }
\ip w, g_Q, \gamma _Q . L ^{1,\infty } (dx).  \lesssim w (Q_1) \,. 
\end{equation}
This is a uniform statement in $ Q_1$.  If in addition $ Q_1\in \mathscr P_{n,\alpha}(L)$, we have 
\begin{equation} \label{e.weak}
\NOrm \sum _{\substack {Q \in Q_1\\ Q\in \mathscr P_{n,\alpha}(L)  }} 
\ip w, g_Q, \gamma _Q . L ^{1,\infty } (dx).  \lesssim  2 ^{-\alpha } \frac {w (L)} {\lvert  L\rvert } \cdot \lvert  Q_1\rvert  \,. 
\end{equation}    
Due to the functions $ g_Q$ and $ \gamma _Q$ are supported on $ Q$, we see that this estimate also holds 
uniformly in $ Q_1$.  

Note that we have by the definition of  Haar functions Definition~\ref{d.haar}, and 
a simple Haar shift, Definition~\ref{d.simple}, 
\begin{equation}   \label{e.7single}
\lvert \ip w, g_Q, \gamma _Q (x) \rvert 
\le \frac {w (Q)} {\lvert  Q\rvert } \lesssim 2 ^{-\alpha } \frac {w (L)} {\lvert  L\rvert }\,. 
\end{equation}
The point of these observations is that 
Lemma~\ref{l.JN} applies.  Define 
\begin{equation} \label{e.<}
E _{\alpha }(t) \coloneqq  \Biggl\{  x \in L \;:\; 
\Biggl\lvert 
\sum _{\substack{  Q\in \mathscr P_{n,\alpha}(L)  }} \ip w, g_Q, \gamma _Q (x) 
\Biggr\rvert
>  K t  2 ^{-\alpha }\frac {w (L)} {\lvert  L\rvert } \Biggr\}\,, \qquad t\ge 1 \,.  
\end{equation}
We have the exponential inequality 
$
\lvert  E _{\alpha } (t)\rvert \lesssim  \operatorname e ^{-t}   \lvert L\rvert  
$
for an appropriate choice of constant $ K$ in \eqref{e.<}.  (The choice of $ K$ is dictated only by the 
exact constants that enter into \eqref{e.weak} and \eqref{e.7single} as well as the parameter $ \tau $ 
associated with the simple Haar shift.)

This is one of our two claims, the distributional estimate in Lebesgue measure \eqref{e.ess1}, for the collection
$ \mathscr P_{n,\alpha}(L)$, not 
the collection $ \mathscr P_n (L)$.  But with the term $ 2 ^{-\alpha } $  appearing in \eqref{e.<}, it is 
easy to supply \eqref{e.ess1} as written.  Indeed, for $ K' = K \sum _{\alpha } 2 ^{-\alpha /2}$, and 
$ K$ as in \eqref{e.<}, we can estimate 
\begin{align*}
\Biggl\lvert 
\Biggl\{  x \in L \;:\; 
\Biggl\lvert 
\sum _{\substack{  Q\in \mathscr P_{n}(L)  }} \ip w, g_Q, \gamma _Q (x) 
\Biggr\rvert
>  K' t  \frac {w (L)} {\lvert  L\rvert } \Biggr\}
\Biggr\rvert 
& \le \sum _{\alpha =0} ^{\infty } \lvert  E _{\alpha } (t 2 ^{\alpha/2})\rvert \lesssim e ^{-t} \lvert  L\rvert.   
\end{align*}
\medskip 

We want the corresponding inequality in $ w ^{-1} $-measure.  But note that $ E _{\alpha }(t)$ is 
a union of disjoint dyadic cubes  in a collection $  \mathscr E _{\alpha } (t)$,
where for each $ Q\in \mathscr E  _{\alpha }(t)$, 
we can choose dyadic $ \phi (Q)\in  \mathscr P_{n,\alpha}(L)$ 
with  $ Q\subset \phi (Q)$, and 
$ \lvert  Q\rvert\ge 2 ^{-\tau d} \lvert  \phi (Q)\rvert  $. This follows from  the definition of a simple Haar shift.  
It follows that we have 
\begin{equation} \label{e.qx} 
\Abs { \textstyle\bigcup \{\phi (Q) \;:\; Q \in \mathscr E _{\alpha } (t)\} } 
\lesssim \operatorname e ^{-t} \lvert  L\rvert  \,. 
\end{equation}
(Recall that there is a similar difficulty in Proposition~\ref{p.weakL1}.) 
The point of these considerations is this:  For each $ Q'\in \mathscr P_{n,\alpha}(L)$, we have 
both the equivalences \eqref{e.QnDef} and \eqref{e.al}.  Hence,  $ w ^{-1} (Q') \simeq \rho \lvert  Q'\rvert $ 
where $ \rho $ is a fixed quantity.  (It depends upon $ L$, and we can 
we can compute it, but as it appears on both sides of the distributional inequality, its value 
is irrelevant to our conclusion.)
We can conclude from \eqref{e.qx} 
the same inequality in $ w ^{-1} $-measure by the following argument.  Let $ \mathscr E _{\alpha }  ^{\ast} (t)$ 
\begin{align*}
w ^{-1} \Biggl\{ 
\Biggl\lvert 
\sum _{Q \in \mathscr P_{n,\alpha } (L)} \ip w , g_Q, \gamma _Q 
\Biggr\rvert
>  K t 2 ^{-\alpha } 
\frac {w (L)} {\lvert  L\rvert }\Biggr\} 
& \le 
w ^{-1} \bigl({ \textstyle\bigcup \{\phi (Q) \;:\; Q \in \mathscr E _{\alpha } (t)\} } \bigr) 
\\
& = 
\sum _{Q\in \mathscr E _{\alpha }  ^{\ast} (t)} w ^{-1} (\phi (Q))
\\
& \simeq \rho 
\sum _{Q\in \mathscr E _{\alpha }  ^{\ast} (t)} 
\lvert  \phi (Q) \rvert 
\\
& \lesssim 
\rho 
\Abs { \textstyle\bigcup \{\phi (Q) \;:\; Q \in \mathscr E _{\alpha } (t)\} } 
\\
&  \lesssim \rho \operatorname e ^{-t}  \lvert  L\rvert  \simeq e ^{-t} w ^{-1} (L)
\,. 
\end{align*}
  This  \eqref{e.ess2},  except for the occurence  of the $ 2 ^{-\alpha }$ on the right, 
  and so the proof is complete.

%%%%%%%%%%%%%%%%%%%%%%%%%%%%%% SECTION  SECTION SECTION
%%%%%%%%%%%%%%%%%%%%%%%%%%%%%% SECTION  SECTION SECTION 
\section{Sufficient Conditions for a Two Weight Inequality} %\label{s.}

There are a great many sufficient conditions for a two-weight inequality.  To these 
results, let us add this statement, for it's elegance. (It is probably  already known.)

%%%%%%%%%%%%%%%%%%%%%%%%%%%%%% THEOREM THEOREM THEOREM
\begin{theorem}\label{t.suff}  Let $ \alpha , \beta $ be positive functions on $ \mathbb R ^{d}$. 
For the inequality below to hold for all Haar shift operators $\operatorname T $ 
\begin{equation} \label{e.T2}
\norm \operatorname T (f \alpha ). L ^{2} (\beta ).  \lesssim \norm f. L ^{2} (\alpha ).  
\end{equation}
It is sufficient that $ \alpha ,\beta \in A _{\infty }$ and the following `two-weight $ A_2$' hold:
\begin{equation}\label{e.2A2}
\sup _{Q} \frac {\alpha (Q)} {\lvert  Q\rvert } \cdot 
\frac{\beta  (Q)} {\lvert  Q\rvert } < \infty \,. 
\end{equation}
\end{theorem}
%%%%%%%%%%%%%%%%%%%%%%%%%%%%%% THEOREM THEOREM THEOREM

Of course these conditions are not necessary, for example one can take $ \alpha = \beta = \mathbf 1_{E}$, 
for any measurable subset $ E$ of $ \mathbb R ^{d}$.  
By $ \alpha \in A _{\infty }$ we mean the measures $ \alpha $ and $ \beta $ satisfy a variant of the 
estimate in \eqref{e.AzI}.  
%%%%%%%%%%%%%%%%%%%%%%%%%%%%%%  DEFINITION DEFINITION DEFINITION
\begin{definition}\label{d.Ainfty} We say that measure $ \alpha \in A _{\infty }$ if 
this condition holds.  For all $ 0<\epsilon <1$ there is a $ 0< \eta <1$ so that for all cubes $ Q$ 
and sets $ E\subset Q$ with $ \lvert  E\rvert< \epsilon \lvert  Q\rvert  $, then $ \alpha (E)< \beta \alpha (Q)$. 
\end{definition}
%%%%%%%%%%%%%%%%%%%%%%%%%%%%%%  DEFINITION DEFINITION DEFINITION

The proof is a modification of what we have already presented, so we do not give the details. 
The resulting estimate is however sharp in the dependence upon the two weight $ A_2$ constant, and the $ A _{\infty }$ 
constants.

\end{document}